\documentclass[12pt,a4paper,reqno]{amsart}

\usepackage[utf8]{inputenc}
\usepackage[main=british,ngerman]{babel}
\usepackage{amsmath,amsthm,amssymb,mathtools}
\mathtoolsset{centercolon}
\usepackage{hyperref,float,caption}
\hypersetup{
    colorlinks,
    linkcolor={red!60!black},
    citecolor={green!60!black},
    urlcolor={blue!60!black},
}
\usepackage{geometry}
\usepackage[nameinlink]{cleveref}
\usepackage{enumitem,xspace}
\usepackage[dvipsnames]{xcolor}
\usepackage[abbrev, msc-links, alphabetic, nobysame]{amsrefs}

\usepackage[labelformat=simple]{subcaption}
\usepackage{tikz}

\newtheorem{theorem}{Theorem}[section]
\newtheorem{lemma}[theorem]{Lemma}
\theoremstyle{remark}
\newtheorem{claim}{Claim}

\def\lqedsymbol{\ifmmode$\lrcorner$\else{\unskip\nobreak\hfil
\penalty50\hskip1em\null\nobreak\hfil$\rule{1.2ex}{1.2ex}$
\parfillskip=0pt\finalhyphendemerits=0\endgraf}\fi} 

\newenvironment{claimproof}[1][\proofname]
  {%
    \proof[#1]%
  }
{%
    \endproof%
  }

\makeatletter
\@addtoreset{claim}{theorem}
\makeatother

\crefname{section}{section}{sections}
\crefformat{section}{#2Section~#1#3}
\Crefformat{section}{#2Section~#1#3}

\crefname{lemma}{lemma}{lemmata}
\crefformat{lemma}{#2Lemma~#1#3}
\Crefformat{lemma}{#2Lemma~#1#3}

\crefname{theorem}{theorem}{theorems}
\crefformat{theorem}{#2Theorem~#1#3}
\Crefformat{theorem}{#2Theorem~#1#3}

\crefname{figure}{figure}{figures}
\crefformat{figure}{#2Figure~#1#3}
\Crefformat{figure}{#2Figure~#1#3}

\crefname{claim}{claim}{claims}
\crefformat{claim}{#2Claim~#1#3}
\Crefformat{claim}{#2Claim~#1#3}

\newcommand{\braces}[1]{\lbrace #1 \rbrace}
\newcommand{\mcalH}{\mathcal{H}}

\newcommand{\GF}{G_\text{FP}}
\newcommand{\GS}{G_\text{SP}}
\newcommand{\Kalephzero}{K^{\aleph_0}}
\newcommand{\kthree}{K^3}
\newcommand{\kfour}{K^4}
\newcommand{\kvimi}{K^4_{-}}
\newcommand{\kfive}{K^5}

\newcommand{\fp}{the first player\xspace}
\newcommand{\secp}{the second player\xspace}
\newcommand{\Secp}{The second player\xspace}
\newcommand{\fps}{the first player's\xspace}
\newcommand{\secps}{the second player's\xspace}
\newcommand{\FP}{The first player\xspace}
\newcommand{\fpdegree}{$F$-degree\xspace}
\newcommand{\secpdegree}{$S$-degree\xspace}

\newcommand{\SetOfPi}{P}

\newcommand{\fpmain}{$\fpmainInMath$\xspace}
\newcommand{\secpmain}{$\secpmainInMath$\xspace}
\newcommand{\secpleft}{$\secpleftInMath$\xspace}
\newcommand{\secpright}{$\secprightInMath$\xspace}

\newcommand{\fpmainInMath}{m}
\newcommand{\secpmainInMath}{n}
\newcommand{\secpleftInMath}{\ell}
\newcommand{\secprightInMath}{r}

\colorlet{FreeEdgeColour}{green!60!black}
\colorlet{FirstPlayerColour}{red}
\colorlet{DesiredMoveColour}{green!60!black}
\colorlet{DesiredAlternativeMoveColour}{green!60!black}

\title{The $K^4$-game}
\author[N.\ Bowler]{Nathan Bowler}
\author[F.\ Gut]{Florian Gut}

\address{Universit\"at Hamburg, Department of Mathematics, Bundesstrasse 55 (Geomatikum), 20146 Hamburg, Germany}
\email{\{nathan.bowler, florian.gut\}@uni-hamburg.de}

\begin{document}

\begin{abstract}
    We investigate a two player game called the $K^4$-building game: two players alternately claim edges of an infinite complete graph.
    Each player’s aim is to claim all six edges on some vertex set of size four for themself.
    The first player to accomplish this goal is declared the winner of the game.
    We present a winning strategy which guarantees a win for the first player.
\end{abstract}

\maketitle

\section{Introduction}\label{sec:introduction}

While games of incomplete information are studied in economics, games of complete information are of great interest for mathematicians.
This field was seeded by very simple well known games such as tic-tac-toe or hex, sparking the invention of new games up to arbitrarily complex ones.
\emph{Game of complete information} means that all players have all the relevant information about the game at any point.
Thus analysing such games often comes down to investigating every possible board state and the resulting number of plays of the game that must be considered.
This prompted the term \emph{combinatorial games} for these kinds of games.
We refer to Beck's book about combinatorial games \cite{B08} for a comprehensive overview.

We study a particular two player game played by the \emph{first player} and the \emph{second player}.
Beginning with \fp they alternately colour exactly one uncoloured edge of a sufficiently large complete graph $B$, the \emph{board}, in their respective colour.
The players agree upon some graph $G$ before the start of the game and their respective aim is to be the first to have a copy of $G$ contained in their coloured subgraph.
We call this the $(B,G)$\emph{-game} or the $G$\emph{-building game} if the board does not matter or is clear from context.

This type of game is closely connected to the area of Ramsey theory.
Let $R(k)$ stand for the smallest number $n$ of vertices such that for any colouring of $K^n$ with two colours there is a monochromatic $K^k$ contained as a subgraph.
By Ramsey's theorem (see e.g. \cite{D17}*{Theorem 9.1.1}), in the $G$-building game if the board is a complete graph of size at least $n \geq R(\vert G\vert)$ then after all the edges of the board have been claimed at least one of the players' graphs must contain a copy of $G$.
Thus, as long as $B$ is finite, Ramsey's theorem implies the existence of a winning strategy for one of the players. 
Furthermore by a game theoretic argument called \emph{strategy stealing} (see e.g. \cite{HKSS14} by Hefetz, Krivelevich, Stojakovi\'c, and Szab\'o) this must be \fp.
While this argument proves the existence of a winning strategy for \fp, there is no information about what such a strategy looks like.

If instead the countably infinite complete graph $\Kalephzero$ is chosen as the board, very different considerations come into play.
Since the edges of the board need not be exhausted in this version of the game, Ramsey's theorem can no longer be applied, i.e. it may happen that neither of the players finishes his or her copy of $G$.
In that case \secp is declared the winner of the game.
One can imagine this in such a way that \secp quickly builds some substructure and then forces the first player to make particular non-winning moves from this finite point on in order to prevent him from winning.
This is not at all far fetched as shown by Hefetz, Kusch, Narins, Pokrovskiy, Requil\'e and Sarid in~\cite{HKNPRS17} where they investigate a $5$ uniform version of this game:
The board is the infinite complete $5$ uniform hyper graph and Hefetz et al prove the existence of a $5$ uniform hyper graph $\mcalH$ for which the second player can delay the game indefinitely before the first player can finish her copy of $\mcalH$.
This illustrates that an infinite board also has implications for the strategy stealing argument.
The reason for this is that the goal of the game are no longer symmetric for the two players.
\Secp will be happy if he manages to delay the game indefinitely, but if \fp steals this strategy, she is not declared the winner of the game.

The deliberations above imply that when the board is infinite our best bet is to find an explicit strategy for \fp.
Such a strategy can then also be used in the $(B,G)$-game so long as $B$ has sufficient size.
In this paper we analyse the $(\Kalephzero,\kfour)$-game and indeed, there is a winning strategy for \fp which we present in \cref{sec:K_4_game}.
The strategy is also a winning strategy for \fp in the $(K^n,\kfour)$-game if $n$ is at least $17$, thus in particular there is a winning strategy for \fp in the $(K^{17},\kfour)$-game.
Therefore there is a winning strategy on a board that is smaller than suggested by Ramsey theory, as $R(4)=18$ which was proved by Greenwood and Gleason in \cite{GG55}*{Section 3}.

The winning strategy presented in \cref{sec:K_4_game} draws on ideas first introduced by Beck in \cite{B02}*{Section 5} in 2002.
More precisely, our strategy broadly follows the one presented in \cite{B02}.
Unfortunately, the latter pools together some cases that must be considered separately and, more importantly, claims that some cases must not be investigated further for brevity but in fact there are ways to win for \secp in these cases.
We examine these issues in \cref{sec:original_proof}.
Beck noticed these shortcomings and therefore posed finding an explicit strategy for the $(K^{R(4)},\kfour)$-game as a question in \cite{B08}*{Open Problem 4.6} in 2008, which we answer in this paper.
Our proof deals with all cases separately, which means there are a large number of cases, thus we use a computer algorithm to validate \fps strategy in all of them.
This allows us to also deduce that it takes \fp at most 21 turns to win the game.

The graph $G$ that the players aim to build in the presented strong game is very small, it only has four vertices.
It appears futile to use a similar approach to find a winning strategy in $G$-building games for larger graphs $G$, even in the $(\Kalephzero,\kfive)$-game, thus new techniques must be found to give an explicit winning strategy in the $\kfive$-game.
Another possible twist in the $(\Kalephzero,G)$-game is to also allow $G$ to be infinite.
Recent work has demonstrated that this is an interesting field of inquiry with many open questions, see \cite{BEG23}.

\section{Preliminaries}\label{sec:preliminaries}

Throughout this paper we will draw on the definitions as established in \cite{D17}, if not explicitly mentioned otherwise.
We will assume that \fp chooses the colour red and refer to that player as \emph{she} or \emph{her}.
Likewise we will assume that \secp chooses the colour blue and refer to that player as \emph{he} or \emph{him}.
By a \emph{turn} of a player we mean that that player chooses an edge of the board that has not been coloured by either of the players and colours that edge in her or his respective colour.
We will alternatively also call a turn a \emph{move} interchangeably and also say that that player \emph{claims} an edge.
A sequence of one move by \fp and then one by \secp is called a \emph{round}.
The turns are enumerated by their chronological order and so are the rounds.
Thus after the \emph{first round}, there are precisely two coloured edges on the board, one edge of \fp and one of \secp.
By the $n$\emph{-th} turn or the $n$\emph{-th} move of \fp we mean the $n$-th time she claims an edge and similarly for the $n$\emph{-th} turn and $n$\emph{-th} move of \secp.
When we say that a player \emph{connects} a vertex $v$ to a vertex $w$, we mean that that player claims the edge $vw$.
We will mean the same when we say that a player \emph{plays from $v$ to $w$}.
We define $E(\GF)$ to be the edges that \fp has claimed up to that point and $V(\GF)$ are the vertices of the board that are incident with at least one edge of $E(\GF)$.
With this we define $\GF := (V(\GF),E(\GF))$.
We define $E(\GS)$, $V(\GS)$ and $\GS$ similarly for \secp.
If we write \emph{a fresh vertex} we mean a vertex $v \in  V \setminus V(\GF \cup \GS) $.
For a vertex $v$ of the board we define the \emph{\fpdegree} of $v$ to be the degree of $v$ in $\GF$ if $v \in \GF$ and $0$ otherwise.
Similarly for the \emph{\secpdegree} and $\GS$.
By $\kvimi$ we denote the unique graph on $4$ vertices with $5$ edges.
A \emph{threat} is a monochromatic induced subgraph of $\GF \cup \GS$ that is  isomorphic to $\kvimi$.
Based on this we call a graph $H$ a \emph{threat seed graph} or simply a \emph{threat seed} if it is a monochromatic induced subgraph of $\GF \cup \GS$ and it is a graph with precisely four vertices and four edges.
That is, it is isomorphic to either a cycle of length four or a triangle with an attached edge.

To keep figures as clear as possible we always just depict the edges which have been claimed by either player as well as their incident vertices.
Additionally, if a figure illustrates a possible course of the game up to some turn, we indicate the order in which edges in that figure have been claimed by numbers that are depicted on the edges.
For $i < j$ the edge depicted with $i$ is claimed before the one depicted with $j$.

\section{Problems with an earlier strategy}\label{sec:original_proof}

In the proof of \cite{B02}*{Theorem A.1} the discussion of case ``L(3)'' has a flaw, the assertion ``The graph of the first five blue edges always forms a path or two separated paths with some red edges between them.'' does not hold true.
Note that in \cite{B02} the colour blue refers to \secp and red refers to \fp, as in this paper.
We present a course of the game that falls under L(3) and in which after five turns $\GS$ is a tree with three leaves.
This has a large enough impact that if \fp follows the strategy given in \cite{B02} then \secp can win the game, as illustrated in \cref{fig:SP_can_win_L(3)}.
\begin{figure}[ht]
    \centering
    \begin{tikzpicture}[scale=.75]
        \foreach \coordinate/\name in {(1, 1)/A,(1, -1)/B,(-2, 0)/C,(4, 0)/D,(-.5, -3)/K,(2.5, -3)/L,(-2, 6)/P1,(1, 8)/P2,(4, 6)/P3,(6, 4)/P4} \node[shape=coordinate] at \coordinate (\name) {};
        \node[left] at (A) {$A$};
        \node[below] at (B) {$B$};
        \node[left] at (C) {$C$};
        \node[right] at (D) {$D$};
        \node[below] at (K) {$K$};
        \node[below] at (L) {$L$};
        \node[above left] at (P1) {$P_1$};
        \node[above] at (P2) {$P_2$};
        \node[above] at (P3) {$P_3$};
        \node[above] at (P4) {$P_4$};
        \draw[red, very thick] 
            (A) to node[fill=white,inner sep = 1pt,rounded corners] {$0$} (C)
            (A) to node[fill=white,inner sep = 1pt,rounded corners,pos=.52] {$2$} (K)
            (B) to node[fill=white,inner sep = 0.8pt,rounded corners, pos=.7] {$4$} (A)
            (C) to node[fill=white,inner sep = 1pt,rounded corners] {$6$} (B)
            (B) to node[fill=white,inner sep = 1pt,rounded corners,pos=.4] {$8$} (D)
            (A) to node[fill=white,inner sep = 1pt,rounded corners] {$10$} (D)
            (A) to node[fill=white,inner sep = 1pt,rounded corners] {$12$} (P1)
            (D) to node[fill=white,inner sep = 1pt,rounded corners, near end] {$14$} (P1)
            (A) to node[fill=white,inner sep = 1pt,rounded corners] {$16$} (P2)
            (A) to node[fill=white,inner sep = 1pt,rounded corners,pos=.45] {$18$} (P3)
            (A) to node[fill=white,inner sep = 1pt,rounded corners] {$20$} (P4)
            (L) to node[fill=white,inner sep = 1pt,rounded corners] {$22$} (D)
            (P1) to [out=285,in=130] node[fill=white,inner sep = 1pt,rounded corners] {$24$} (L)
            (P3) to [out=228,in=90] node[fill=white,inner sep = 1pt,rounded corners, pos=.25] {$26$} (K);
        \draw[blue, very thick]
            (L) to node[fill=white,inner sep = 1pt,rounded corners] {$1$} (A)
            (C) to node[fill=white,inner sep = 1pt,rounded corners] {$3$} (K)
            (K) to node[fill=white,inner sep = 1pt,rounded corners] {$5$} (B)
            (K) to node[fill=white,inner sep = 1pt,rounded corners] {$7$} (L)
            (C) to node[fill=white,inner sep = 1pt,rounded corners, near end] {$9$} (D)
            (C) to node[fill=white,inner sep = 1pt,rounded corners, near end] {$11$} (L)
            (C) to node[fill=white,inner sep = 1pt,rounded corners] {$13$} (P1)
            (B) to node[fill=white,inner sep = 1pt,rounded corners] {$15$} (P1)
            (B) to [out=80,in=280] node[fill=white,inner sep = 1pt,rounded corners, near end] {$17$} (P2)
            (B) to node[fill=white,inner sep = 1pt,rounded corners] {$19$} (P3)
            (K) to node[fill=white,inner sep = 1pt,rounded corners, near end] {$21$} (D)
            (K) to node[fill=white,inner sep = 1pt,rounded corners] {$23$} (P1)
            (P1) to node[fill=white,inner sep = 1pt,rounded corners, near start] {$25$} (P3)
            (P1) to node[fill=white,inner sep = 2pt,rounded corners] {$27$} (P2);
        \foreach \x in {A,B,C,D,K,L,P1,P2,P3,P4} \draw[fill] (\x) circle [radius=.09];
    \end{tikzpicture}
    \caption{A course of the game in case ``L(3)'' in the proof of \cite{B02}*{Theorem A.1} in which \secp can win the game. Note that the edges labelled \textcolor{red}{$22$}, \textcolor{red}{$24$} and \textcolor{red}{$26$} are forced moves for \fp in that they are reactions to threats by \secp. After claiming the edge $P_1P_2$, \secp can win by either claiming $P_2 P_3$ or $K P_2$ with his next move.}
    \label{fig:SP_can_win_L(3)}
\end{figure}
Note that a similar problem arises if \fp claims the edges $BP_x$ instead of $AP_x$ in her turns.
Thus, in case $\GF \cup \GS$ is isomorphic to the graph mentioned in L(3) or a similar one (see \cref{fig:stage_6_special_cases}), \fp must instead divert from her usual course of play.
We present a possible strategy for this special case in \labelcref{item:stage_six} of the $\kfour$ building strategy presented in \cref{sec:K_4_game}.

The case ``L(4)'' also has an issue.
The second to last paragraph of the proof of \cite{B02}*{Theorem A.1} asserts that in case L(4) if there is no blue triangle either containing the vertex $B$ or the vertex $C$ after six turns, then \fp can continue with the standard strategy of the proof and win.
This is not true.
In \cref{fig:SP_can_win_L(4)} we present a possible course of the game falling under case L(4) where after six turns there is no blue triangle at all but \secp wins that game if \fp adheres to her regular strategy.
Thus in that case \fp must use a different strategy for any of the board states depicted in \cref{fig:stage_5_special_cases}.
We define a possible strategy in \labelcref{item:stage_five} of the $\kfour$ building strategy.
\begin{figure}[ht]
    \centering
    \begin{tikzpicture}[scale=.8]
        \foreach \coordinate/\name in {(-2, 0)/A,(0, -3)/K,(1, -1)/C,(4, 0)/D,(1, 1)/B,(2, -3)/L,(1, -5)/M,(-2, 4.5)/P1,(4, 4.5)/P2,(6, 3)/P3} \node[shape=coordinate] at \coordinate (\name) {};
        \node[above] at (A) {$A$};
        \node[above] at (B) {$B$};
        \node[left] at (C) {$C$};
        \node[right] at (D) {$D$};
        \node[below left] at (K) {$K$};
        \node[right] at (L) {$L$};
        \node[below] at (M) {$M$};
        \node[above] at (P1) {$P_1$};
        \node[above] at (P2) {$P_2$};
        \node[right] at (P3) {$P_3$};
        \draw[red, very thick] 
            (A) to node[fill=white,inner sep = 1pt,rounded corners] {$0$} (C)
            (A) to node[fill=white,inner sep = 1.75pt,rounded corners] {$2$} (K)
            (A) to node[fill=white,inner sep = 1pt,rounded corners] {$4$} (B)
            (B) to node[fill=white,inner sep = 1pt,rounded corners, near end] {$6$} (C)
            (B) to node[fill=white,inner sep = 1pt,rounded corners] {$8$} (D)
            (C) to node[fill=white,inner sep = 1pt,rounded corners] {$10$} (D)
            (B) to node[fill=white,inner sep = 1pt,rounded corners] {$12$} (P1)
            (B) to node[fill=white,inner sep = 2pt,rounded corners] {$14$} (P2)
            (B) to node[fill=white,inner sep = 1pt,rounded corners] {$16$} (P3)
            (B) to [out=280,in=80] node[fill=white,inner sep = 1pt,rounded corners,near end] {$18$} (M)
            (C) to node[fill=white,inner sep = 1pt,rounded corners, pos =.675] {$20$} (L)
            (M) to [out=150,in=230] node[fill=white,inner sep = 1pt,rounded corners,pos=.1] {$22$} (P1);
        \draw[blue, very thick]
            (B) to [out=300,in=90] node[fill=white,inner sep = 2pt,rounded corners, pos=.65] {$1$} (L)
            (K) to node[fill=white,inner sep = 1pt,rounded corners, near end] {$3$} (C)
            (K) to node[fill=white,inner sep = 2pt,rounded corners, pos=.3] {$5$} (B)
            (K) to node[fill=white,inner sep = 2pt,rounded corners] {$7$} (M)
            (A) to node[fill=white,inner sep = 1pt,rounded corners,near end] {$9$} (D)
            (M) to node[fill=white,inner sep = 1pt,rounded corners] {$11$} (L)
            (C) to node[fill=white,inner sep = 1pt,rounded corners] {$13$} (P1)
            (C) to node[fill=white,inner sep = 1pt,rounded corners,near end] {$15$} (P2)
            (K) to node[fill=white,inner sep = 0pt,rounded corners, pos=.3] {$17$} (L)
            (C) to node[fill=white,inner sep = 0pt,rounded corners, pos=.3] {$19$} (M)
            (K) to [out=150,in=240] node[fill=white,inner sep = 1pt,rounded corners] {$21$} (P1)
            (K) to [out=90,in=190] node[fill=white,inner sep = 1pt,rounded corners,near end] {$23$} (P2);
        \foreach \node in {A,B,C,D,K,L,M,P1,P2,P3} \draw[fill] (\node) circle [radius=.09];
    \end{tikzpicture}
    \caption{A course of the game in case ``L(4)'' in the proof of \cite{B02}*{Theroem A.1} in which \secp can win the game. Note that the edges labelled \textcolor{red}{18}, \textcolor{red}{20} and \textcolor{red}{22} are moves in which \fp reacts to threats by \secp. In his next move \secp can win by claiming either $P_1P_2$ or $MP_2$.}
    \label{fig:SP_can_win_L(4)}
\end{figure}
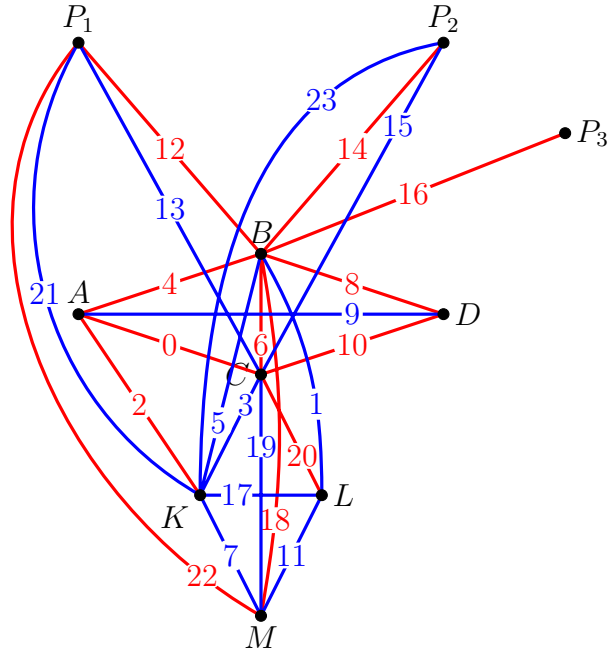

Finally, there is another inaccuracy in the case ``L(4)''
In the subcase where there are two blue triangles, one containing the vertex $B$ and one containing the vertex $C$ of case ``L(4)'', the proof says that \fp should divert from her usual strategy and instead claim $BD_1$, $BD_2$ and $BD_3$ in her sixth, seventh and eighth turn respectively for suitable fresh vertices $D_1$, $D_2$ and $D_3$.
The only course of the game that is considered is the one where \secp reacts by claiming $AD_x$ in each of these three turns, claiming that if he does not do this then \fp can claim $AD_x$ instead and the board is then isomorphic to the board in one of the other cases.
Unfortunately this is not true.
Suppose \secp does not claim $AD_3$.
Then \fp claims $AD_3$ and \secp must claim $CD_3$ in his next turn since that is a threat by \fp.
Under a desired isomorphism each $D_x$ would then need to be mapped to a $P_y$, except for $D_3$ which would be mapped to $D$ (which in turn must also be mapped to a $P_y$.
Suppose \secp had claimed $D_1D_2$ instead of $AD_3$, then that maps to an edge $P_iP_j$ and the proof never considers a situation where there is precisely one blue edge and no red edge on the subgraph of the board induced by $\braces{P_i \colon i \in \braces{1,2,3,4,5}}$.
Moreover, if \fp adheres to the described strategy then there is a way for \secp to win the game.
We deal with this special case in \labelcref{item:stage_five}.

\section{The \texorpdfstring{$\kfour$}{K4}-building game}\label{sec:K_4_game}

We begin by investigating the game in which the player's aim is to construct a $\kthree$, the proof is folklore, we state it here to keep this paper self-contained.
We define the $\kthree$\emph{-building strategy} as follows.
In her first two turns, \fp claims a path of length $2$:
she first claims some edge and calls it $ab$.
If \secp claimed an edge incident with $ab$, by renaming the vertices if necessary, we can assume that it is incident with $b$.
In any case \fp chooses a fresh vertex, calls it $c$ and claims the edge $bc$.
If \secp does not claim $ac$ in his second turn, \fp does so in her third turn and thereby finishes a monochromatic triangle on $a$, $b$ and $c$.
Otherwise, if \secps first edge $uv$ was disjoint from her first edge $ab$ and she cannot finish her triangle in three turns, then she claims $bu$ in her third turn.
If \secps first edge was not disjoint from \fps first edge, then it is incident with $b$ by assumption.
In this case \fp chooses a fresh vertex $u$ and colours $bu$ in her third turn.
In her fourth turn at least one of the edges $au$, $cu$ are unclaimed and so she can complete a $\kthree$ on at least one of $a$, $b$ and $u$ or $a$, $c$ and $u$.

Note that \fp needs at most five vertices to execute her strategy with the only condition that there are no claimed edges in between these vertices.
Note further that the dependence of the third move of \fp on the first move of \secp is not necessary for \cref{lem:triangle} but it is for \cref{theo:k_four_game}, thus we include it here already.

\begin{lemma}\label{lem:triangle}
    The $\kthree$ building strategy is a winning strategy for \fp in the $\kthree$ game.
    Moreover, \fp needs at most $4$ turns to win the $\kthree$ game with the $\kthree$ building strategy.
\end{lemma}

\begin{proof}
        The moreover part follows from the definition of the strategy.
	If it takes \fp only three turns to construct her $\kthree$, she wins the game, as the second player has only had two turns up to that point.
	If it takes her four turns, then \secp colours $ac$ in his second turn.
	Let $b_1$ denote the edge \secp claimed in his first turn.
        Either $b_1$ is disjoint from \fps first edge $ab$ or it is incident with $b$.
        Thus it is not incident with $a$.
	As \fp chose $c$ as a fresh vertex in her second turn, $b_1$ is also not incident with $c$.
	So $ac$ and $b_1$ have no vertices in common.
	Thus, the first two blue edges already use four vertices and therefore can not be two edges of a triangle.
\end{proof}

Before we give the complete strategy for \fp for the $\kfour$ game, we need to make some adaptions to the $\kthree$ building strategy, which we call the \emph{modified} $K^3$ \emph{building strategy}.
This strategy will be used on multiple occasions by \fp in her strategy for the $\kfour$ game to build new triangles in particular regions of the board.

In principle, the modified $\kthree$ building strategy consists of the same steps as the $\kthree$ building strategy but in the details \fp is more cautious, since there may be more coloured edges on the board at later points of the game.
We define $\SetOfPi := \braces{p_0,p_1,p_2,p_3,p_4}$.
\begin{itemize}
    \item \fps first edge is an edge that is incident with two vertices of $\SetOfPi$ that have smallest possible \secpdegree.
    \item For \fps second edge, which we call $f$,
    \begin{itemize}
        \item If \secps previous edge is incident with a vertex of $f$ then \fps next edge is incident with that vertex and a vertex of $\SetOfPi$ that is fresh in $G[\SetOfPi]$ with the smallest possible \secpdegree.
        \item If \secps previous edge is incident with two vertices of $\SetOfPi$ and not with a vertex of $f$ then \fps next edge is incident with a vertex of $f$ with smallest possible \secpdegree and a vertex of \secps edge with smallest possible \secpdegree.
        \item Otherwise, if not all vertices of $\SetOfPi\setminus f$ have the same \secpdegree, \fp claims an edge that is incident with a vertex of $f$ and a vertex of $\SetOfPi\setminus f$ of minimum \secpdegree that is not incident with \secps previous edge.
        \item If all vertices of $\SetOfPi\setminus f$ have the same \secpdegree, \fp claims an edge that is incident with a vertex of $f$ and a vertex of $\SetOfPi\setminus f$ that is also neighbouring as many vertices of \secpleft, \secpright as possible in the coloured subgraph.
    \end{itemize}
    \item For \fps third edge,
    \begin{itemize}
        \item if \fp can finish a triangle with her third edge then she does so.
        \item Otherwise, let $v$ be the unique vertex of \fpdegree 2 on the subgraph induced by $\SetOfPi$ and let $u$ and $w$ be the vertices of \fpdegree 1 in the subgraph induced by $\SetOfPi$ and let $a$ and $b$ be the remaining two vertices of $\SetOfPi$. Then for one $x \in \braces{a,b}$ the edges $xv$, $xu$ and $xw$ are unclaimed.
        \FP claims $xv$ for such a vertex $x$.
    \end{itemize}
    \item If there is a fourth turn, then in that turn \fp can finish a monochromatic triangle on $\SetOfPi$.
\end{itemize}

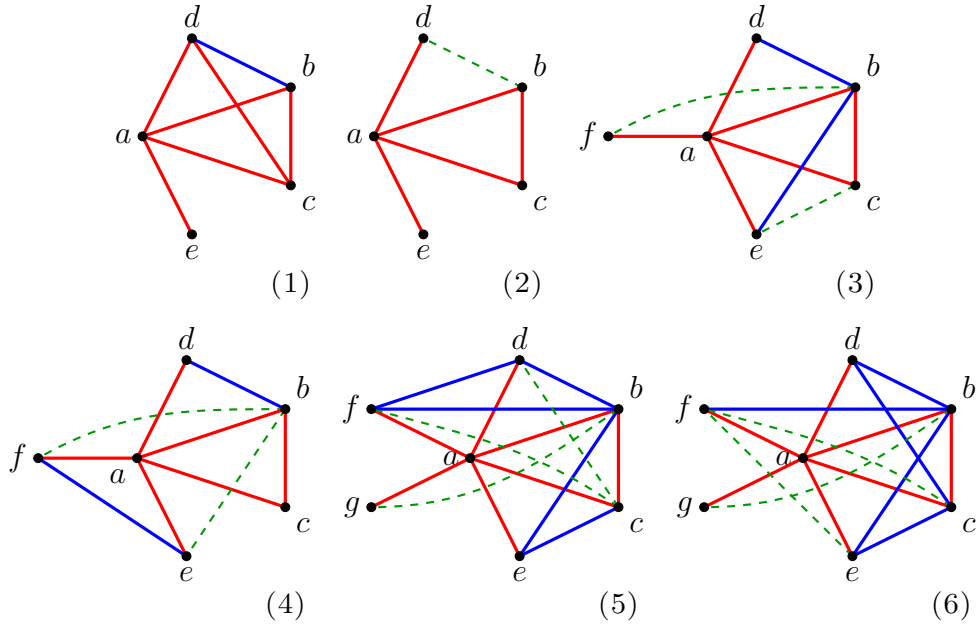
\begin{figure}[ht]
    \centering
    \begin{subcaptiongroup}
    \subcaptionlistentry{Promising graph type 1}
    \label{prom_graph:type_1}
    \begin{tikzpicture}[scale=.65]
        \node at (1,-3) {\captiontext*{}};
        \foreach \coordinate/\name in {(-2, 0)/1,(1, 1)/2,(1, -1)/3,(-1, 2)/4,(-1, -2)/5} \node[shape=coordinate] at \coordinate (\name) {};
        \foreach \x in {2,3,4,5} \draw[red, very thick]  (1) to (\x);
        \foreach \x in {2,4} \draw[red, very thick]  (\x) to (3);
        \node[left] at (1) {$a$};
        \node[below right] at (3) {$c$};
        \node[above right] at (2) {$b$};
        \node[above] at (4) {$d$};
        \node[below] at (5) {$e$};
        \draw[blue, very thick]  (2) to (4);
        \foreach \x in {3,2,1,4,5} \draw[fill] (\x) circle [radius=.09];
    \end{tikzpicture}
    \subcaptionlistentry{Promising graph type 2}
    \label{prom_graph:type_2}
    \begin{tikzpicture}[scale=.65]
        \node at (1,-3) {\captiontext*{}};
        \foreach \coordinate/\name in {(-2, 0)/1,(1, 1)/2,(1, -1)/3,(-1, 2)/4,(-1, -2)/5} \node[shape=coordinate] at \coordinate (\name) {};
        \node[left] at (1) {$a$};
        \node[above right] at (2) {$b$};
        \node[below right] at (3) {$c$};
        \node[above] at (4) {$d$};
        \node[below] at (5) {$e$};
        \draw[thick, dashed, DesiredMoveColour] (2) to (4);
        \foreach \x in {2,3,4,5} \draw[red, very thick]  (1) to (\x);
        \draw[red, very thick]  (2) to (3);
        \foreach \x in {1,2,3,4,5} \draw[fill] (\x) circle [radius=.09];
    \end{tikzpicture}
    \subcaptionlistentry{Promising graph type 3}
    \label{prom_graph:type_3}
    \begin{tikzpicture}[scale=.65]
        \node at (1,-3) {\captiontext*{}};
        \foreach \coordinate/\name in {(-2, 0)/1,(1, 1)/2,(1, -1)/3,(-1, 2)/4,(-1, -2)/5,(-4, 0)/6} \node[shape=coordinate] at \coordinate (\name) {};
        \foreach \x in {2,3,4,5,6} \draw[red, very thick]  (1) to (\x);
        \node[below left] at (1) {$a$};
        \node[above right] at (2) {$b$};
        \node[below right] at (3) {$c$};
        \node[above] at (4) {$d$};
        \node[below] at (5) {$e$};
        \node[left] at (6) {$f$};
        \draw[red, very thick]  (2) to (3);
        \foreach \x in {4,5} \draw[blue, very thick]  (2) to (\x);
        \draw[thick, dashed, DesiredMoveColour] (2) [out=180,in=30] to (6);
        \draw[thick, dashed, DesiredAlternativeMoveColour] (3) to (5);
        \foreach \x in {1,2,3,4,5,6} \draw[fill] (\x) circle [radius=.09];
    \end{tikzpicture}
    
    \subcaptionlistentry{Promising graph type 4}
    \label{prom_graph:type_4}
    \begin{tikzpicture}[scale=.65]
        \node at (1,-3) {\captiontext*{}};
        \foreach \coordinate/\name in {(-2, 0)/1,(1, 1)/2,(1, -1)/3,(-1, 2)/4,(-1, -2)/5,(-4, 0)/6} \node[shape=coordinate] at \coordinate (\name) {};
        \node[below left] at (1) {$a$};
        \node[above right] at (2) {$b$};
        \node[below right] at (3) {$c$};
        \node[above] at (4) {$d$};
        \node[below] at (5) {$e$};
        \node[left] at (6) {$f$};
        \foreach \x in {2,3,4,5,6} \draw[red, very thick]  (1) to (\x);
        \draw[red, very thick]  (2) to (3);
        \draw[blue, very thick]  (2) to (4);
        \draw[blue, very thick]  (5) to (6);
        \draw[thick, dashed, DesiredMoveColour] (2) to (5);
        \draw[thick, dashed, DesiredAlternativeMoveColour] (2) [out=180,in=30] to (6);
        \foreach \x in {1,2,3,4,5,6} \draw[fill] (\x) circle [radius=.09];
    \end{tikzpicture}
    \subcaptionlistentry{Promising graph type 5}
    \label{prom_graph:type_5}
    \begin{tikzpicture}[scale=.65]
        \node at (1,-3) {\captiontext*{}};
        \foreach \coordinate/\name in {(-2, 0)/1,(1, 1)/2,(1, -1)/3,(-1, 2)/4,(-1, -2)/5,(-4, 1)/6,(-4, -1)/7} \node[shape=coordinate] at \coordinate (\name) {};
        \node[left] at (1) {$a$};
        \node[above right] at (2) {$b$};
        \node[below right] at (3) {$c$};
        \node[above] at (4) {$d$};
        \node[below] at (5) {$e$};
        \node[left] at (6) {$f$};
        \node[left] at (7) {$g$};
        \foreach \x in {2,3,4,5,6,7} \draw[red, very thick]  (1) to (\x);
        \draw[red, very thick]  (2) to (3);
        \foreach \x in {4,5,6} \draw[blue, very thick]  (2) to (\x);
        \draw[blue, very thick]  (3) to (5);
        \draw[blue, very thick]  (4) to (6);
        \draw[thick, dashed, DesiredMoveColour] (2) [out=215,in=0] to (7);
        \draw[thick, dashed, DesiredMoveColour] (3) to (4);
        \draw[thick, dashed, DesiredMoveColour] (3) [out=145,in=345] to (6);
        \foreach \x in {1,2,3,4,5,6,7} \draw[fill] (\x) circle [radius=.09];
    \end{tikzpicture}
    \subcaptionlistentry{Promising graph type 6}
    \label{prom_graph:type_6}
    \begin{tikzpicture}[scale=.65]
        \node at (1,-3) {\captiontext*{}};
        \foreach \coordinate/\name in {(-2, 0)/1,(1, 1)/2,(1, -1)/3,(-1, 2)/4,(-1, -2)/5,(-4, 1)/6,(-4, -1)/7} \node[shape=coordinate] at \coordinate (\name) {};
        \node[left] at (1) {$a$};
        \node[above right] at (2) {$b$};
        \node[below right] at (3) {$c$};
        \node[above] at (4) {$d$};
        \node[below] at (5) {$e$};
        \node[left] at (6) {$f$};
        \node[left] at (7) {$g$};
        \foreach \x in {2,3,4,5,6,7} \draw[red, very thick]  (1) to (\x);
        \draw[red, very thick]  (2) to (3);
        \foreach \x in {4,5,6} \draw[blue, very thick]  (2) to (\x);
        \foreach \x in {4,5} \draw[blue, very thick]  (3) to (\x);
        \draw[thick, dashed, DesiredMoveColour] (2) [out=215,in=0] to (7);
        \draw[thick, dashed, DesiredMoveColour] (3) [out=145,in=345] to (6);
        \draw[thick, dashed, DesiredAlternativeMoveColour] (5) to (6);
        \foreach \x in {1,2,3,4,5,6,7} \draw[fill] (\x) circle [radius=.09];
    \end{tikzpicture}
    \end{subcaptiongroup}
    \caption{The six types of promising graphs. In each of the graphs the set of \textcolor{DesiredMoveColour}{dashed} edges represents the set of vulnerable edges. The straight edges represent the edges claimed by \textcolor{red}{\fp} and \textcolor{blue}{\secp}. The non-edges are unclaimed edges.}
    \label{fig:promising_graphs}
\end{figure}
We call a complete graph whose claimed edges induce a graph isomorphic to any of the graphs depicted in \cref{fig:promising_graphs} such that the dashed edges are claimed by \secp, a \emph{promising graph}.
We call the respective dashed edge(s) the \emph{vulnerable edge(s)} of the promising graph of that type.
Additionally, we call a complete graph $H$ that is isomorphic to a promising graph where the vulnerable edges are unclaimed a \emph{realised promising graph} and say that $H$ is \emph{realised} in $G$.
Sometimes we also refer to the intersection of $H$ with $\GF \cup \GS$ as realised promising graph to simplify notation.

\begin{lemma}\label{lem:promising_graphs}
    Let $G$ be a board, $H$ be a promising graph and suppose that $H$ is realised in $G$.
    Suppose that it is \fps turn and further that there is no monochromatic $\kfour$ and no threat by \secp on the board even if the vulnerable edges are added as edges claimed by \secp.
    Then there is a sequence of moves for \fp, each one a threat, the last of which creates two threats, thus \fp wins the game.
\end{lemma}
\begin{proof}
    In \subref{prom_graph:type_1}, by claiming $ce$ \fp immediately creates two threats.
    In \subref{prom_graph:type_2}, by claiming $cd$ \fp creates a threat.
    If \secp reacts to that threat by claiming $bd$ then the resulting subgraph is isomorphic to \subref{prom_graph:type_1}, thus she can create two threats by $ce$.
    In \subref{prom_graph:type_3}, \fp can create a threat by claiming $cf$.
    If \secp reacts to the threat by claiming $bf$ then the induced subgraph on $\braces{a,c,d,e,f}$ is isomorphic to \subref{prom_graph:type_2}, thus \fp can win with the corresponding moves presented in \subref{prom_graph:type_2}.
    Similarly in \subref{prom_graph:type_4}, \fp can create a threat by claiming $ce$ and if \secp reacts by claiming $be$, in her next move \fp creates a threat by claiming $cf$.
    If \secp reacts to the threat by claiming $bf$ then the induced subgraph on $\braces{a,c,d,e,f}$ is isomorphic to the graph in \subref{prom_graph:type_1}, that is \fp can create two threats by claiming $cd$.
    In \subref{prom_graph:type_5}, by claiming $cg$ \fp creates a threat and if \secp reacts by claiming $bg$ then the subgraph induced by $\braces{a,c,d,e,f,g}$ is isomorphic to the one in \subref{prom_graph:type_4}.
    Likewise in \subref{prom_graph:type_6}, by claiming $cg$ \fp creates a threat.
    If \secp reacts with $bg$, then the subgraph induced by $\braces{a,c,d,e,f,g}$ is isomorphic to the one in \subref{prom_graph:type_3}.
    In each case if \secp does not claim the respective edge mentioned then \fp could claim it in her next turn and finish a monochromatic $\kfour$.

    Since there is no threat or monochromatic $\kfour$ by \secp on the board by assumption, he cannot create a threat or a monochromatic $\kfour$ by claiming any of the vulnerable edges by assumption, all of his moves are forced moves and in the last move there are two threats by \fp, \fp wins the game as claimed.
\end{proof}
Note that for any of the promising graphs, the given sequence of moves for \fp also leads to a win for \fp if any of the edges of \secp indicated in blue in \cref{fig:promising_graphs} is unclaimed instead.
In particular this is true for the edge $bd$ in \subref{prom_graph:type_6}, which we use in the algorithm.

In the following, we define the $K^4$ \emph{building strategy}.

\begin{enumerate}[label=\textbf{Stage \arabic*:},ref=Stage \arabic*,align = left,noitemsep]
\item[\textbf{Standing assumption:}]
    At all points we assume that \fp will need to react to any immediate threat by \secp.
    This overrides anything else suggested below.
    She also checks at any point whether she can just finish a monochromatic $\kfour$ in her turn.
    If so, she does this and wins.
    
\item\label{item:stage_one}
    \FP constructs a triangle using the $\kthree$ building strategy.

\item\label{item:stage_two}
    Throughout this stage \fp wants to build a monochromatic $\kvimi$, incorporating the $\kthree$ from \ref{item:stage_one}.
    This is the point of the game where \fp must check whether \secp played in such a way that she cannot adhere to her usual strategy, as mentioned in \cref{sec:original_proof}.
    This is the first thing she checks at the start of each turn in \cref{item:stage_two}.
    \begin{itemize}
        \item \FP checks whether the graph induced by the claimed edges is isomorphic to one of the graphs in \cref{fig:stage_5_special_cases} or \cref{fig:stage_6_special_cases}.
            If that is the case, she switches to \labelcref{item:stage_five} or \labelcref{item:stage_six} respectively.
        \item Otherwise, she checks whether she can claim an edge such that there is a monochromatic $\kvimi$ in her colour on the board.
             If that is the case, she claims such an edge and switches to \ref{item:stage_three}.
        \item Otherwise, let $c$ be the vertex of the monochromatic triangle of \fp with maximum \secpdegree.
            \fp chooses a new vertex $g$ and claims $cg$.
            Then she will be able to connect $g$ to $a$ or $b$ in her following turn.
    \end{itemize}

\item\label{item:stage_three}
Generally in this stage \fp wants to claim five edges, all incident with the same particular vertex of the $\kvimi$ constructed in \ref{item:stage_two}.
More specifically, \fp does this as follows:
\begin{itemize}
\item When \fp is in \ref{item:stage_three} for the first time and she cannot immediately win, then there is a subraph of the board that is a coloured $\kfour$ of which five edges are claimed by \fp and one is claimed by \secp.
\FP assigns roles to the four vertices of that subgraph and sticks to them throughout \ref{item:stage_three} and \ref{item:stage_four}, see \cref{fig:K_four_minus} for an illustration.
\begin{figure}[ht]
    \begin{tikzpicture}[scale=.65]
        \foreach \coordinate/\name in {(-2, 0)/secpleft,(1, -1)/secpmain,(1, 1)/fpmain,(4,0)/secpright} \node[shape=coordinate] at \coordinate (\name) {};
        \node[above] at (fpmain) {\fpmain};
        \node[below] at (secpmain) {\secpmain};
        \node[left] at (secpleft) {\secpleft};
        \node[right] at (secpright) {\secpright};
        \foreach \startvertex/\endvertex in {fpmain/secpright,secpright/secpmain,secpmain/secpleft,secpleft/fpmain,fpmain/secpmain} \draw[red, very thick]  (\startvertex) to (\endvertex);
        \draw[blue, very thick]  (secpleft) to (secpright);
        \foreach \node in {fpmain,secpmain,secpleft,secpright} \draw[fill] (\node) circle [radius=.09];
    \end{tikzpicture}
    \caption{The $\kvimi$ and the vertex names which \fp assigns in \labelcref{item:stage_three}.}
    \label{fig:K_four_minus}
\end{figure}
    \begin{itemize}
        \item We call the two vertices of the $\kfour$ that are incident with \secps edge \secpleft and \secpright, with arbitrary assignment.
        \item Of the remaining two vertices, if one of them is contained within a monochromatic triangle of \secp, call this vertex \fpmain.
        If both are not contained in a triangle of \secp, call one of the two with the biggest \secpdegree \fpmain.
        \item We call the remaining vertex \secpmain.
    \end{itemize}
Then \fp chooses a fresh vertex $p_0$ and claims the edge from \fpmain to $p_0$.
\item In a later turn during \ref{item:stage_three}, let $p_i$, $i \in \braces{0,1,2,3,4}$ be the most recent fresh vertex that \fp chose.
\begin{itemize}
\item If \secp claimed the edge between \secpmain and $p_i$,
\begin{itemize} 
\item if $i < 4$, then \fp chooses a fresh vertex $p_{i+1}$ and claims the edge between \fpmain and $p_{i+1}$,
\item if $i = 4$, then \fp switches to \ref{item:stage_four}.
\end{itemize}
\item If \secp claimed an edge between $p_i$ and one of \secpleft, \secpright, then \fp claims the respective other edge.
\item If \secp claimed none of the edges between $p_i$ and one of \secpmain, \secpleft, \secpright, then \fp claims the edge between $p_i$ and \secpmain.
\end{itemize}
\end{itemize}

\item\label{item:stage_four}
In this stage \fp constructs a triangle on $\SetOfPi$ with the modified $\kthree$ building strategy.
At the end of \ref{item:stage_four} that triangle together with \fpmain is a monochromatic $\kfour$.

\item\label{item:stage_five}

\begin{figure}[ht]
    \centering
    \begin{tikzpicture}[scale=.65]
        \foreach \coordinate/\name in {(-2, 0)/a,(1, -1)/b,(1, 1)/c,(-2, 3)/p0,(-1, -3)/x,(2, -3)/y} \node[shape=coordinate] at \coordinate (\name) {};
        \node[left] at (a) {$a$};
        \node[above right] at (c) {$c$};
        \node[below] at (b) {$b$};
        \node[above] at (p0) {$p_0$};
        \draw[red, very thick]  (a) to node[fill=white,inner sep = 2pt,rounded corners] {$0$} (x);
        \draw[blue, very thick]  (c) to node[fill=white,inner sep = 2pt,rounded corners] {$1$} (y);
        \draw[red, very thick]  (a) to node[fill=white,inner sep = 2pt,rounded corners] {$2$} (b);
        \draw[blue, very thick]  (x) to node[fill=white,inner sep = 2pt,rounded corners] {$3$} (b);
        \draw[red, very thick]  (a) to node[fill=white,inner sep = 2pt,rounded corners] {$4$} (c);
        \draw[blue, very thick]  (x) to node[fill=white,inner sep = 2pt,rounded corners,near end] {$5$} (c);
        \draw[red, very thick]  (c) to node[fill=white,inner sep = 2pt,rounded corners,near end] {$6$} (b);
        \draw[blue, very thick]  (y) to node[fill=white,inner sep = 0pt,rounded corners,pos=.35] {$7$} (x);
        \draw[red, very thick]  (c) to node[fill=white,inner sep = 2pt,rounded corners] {$8$} (p0);
        \draw[blue, very thick]  (a) to node[fill=white,inner sep = 2pt,rounded corners] {$9$} (p0);
        \foreach \node in {a,b,c,p0,x,y} \draw[fill] (\node) circle [radius=.09];
    \end{tikzpicture}
    \begin{tikzpicture}[scale=.65]
        \foreach \coordinate/\name in {(-2, 0)/a,(1, -1)/b,(1, 1)/c,(-2, 3)/p0,(-1, -3)/x,(2, -3)/y,(-2, -3)/z} \node[shape=coordinate] at \coordinate (\name) {};
        \node[left] at (a) {$a$};
        \node[above right] at (c) {$c$};
        \node[below] at (b) {$b$};
        \node[above] at (p0) {$p_0$};
        \draw[red, very thick]  (a) to node[fill=white,inner sep = 2pt,rounded corners] {$0$} (x);
        \draw[blue, very thick]  (c) to node[fill=white,inner sep = 2pt,rounded corners] {$1$} (y);
        \draw[red, very thick]  (a) to node[fill=white,inner sep = 2pt,rounded corners] {$2$} (b);
        \draw[blue, very thick]  (x) to node[fill=white,inner sep = 2pt,rounded corners] {$3$} (b);
        \draw[red, very thick]  (a) to node[fill=white,inner sep = 2pt,rounded corners] {$4$} (c);
        \draw[blue, very thick]  (x) to node[fill=white,inner sep = 2pt,rounded corners,near end] {$5$} (c);
        \draw[red, very thick]  (c) to node[fill=white,inner sep = 2pt,rounded corners,near end] {$6$} (b);
        \draw[blue, very thick]  (z) to node[fill=white,inner sep = 0pt,rounded corners] {$7$} (x);
        \draw[red, very thick]  (c) to node[fill=white,inner sep = 2pt,rounded corners] {$8$} (p0);
        \draw[blue, very thick]  (a) to node[fill=white,inner sep = 2pt,rounded corners] {$9$} (p0);
        \foreach \node in {a,b,c,p0,x,y,z} \draw[fill] (\node) circle [radius=.09];
    \end{tikzpicture}
    \begin{tikzpicture}[scale=.65]
        \foreach \coordinate/\name in {(-2, 0)/a,(1, -1)/b,(1, 1)/c,(-2, 3)/p0,(-1, -3)/x,(2, -3)/y} \node[shape=coordinate] at \coordinate (\name) {};
        \node[left] at (a) {$a$};
        \node[above right] at (c) {$c$};
        \node[below] at (b) {$b$};
        \node[above] at (p0) {$p_0$};
        \draw[red, very thick]  (a) to node[fill=white,inner sep = 2pt,rounded corners] {$0$} (x);
        \draw[blue, very thick]  (c) to node[fill=white,inner sep = 2pt,rounded corners] {$1$} (y);
        \draw[red, very thick]  (a) to node[fill=white,inner sep = 2pt,rounded corners] {$2$} (b);
        \draw[blue, very thick]  (x) to node[fill=white,inner sep = 2pt,rounded corners] {$3$} (b);
        \draw[red, very thick]  (a) to node[fill=white,inner sep = 2pt,rounded corners] {$4$} (c);
        \draw[blue, very thick]  (x) to node[fill=white,inner sep = 2pt,rounded corners,near end] {$5$} (c);
        \draw[red, very thick]  (c) to node[fill=white,inner sep = 2pt,rounded corners,near end] {$6$} (b);
        \draw[blue, very thick]  (b) to node[fill=white,inner sep = 2pt,rounded corners,pos=.45] {$7$} (y);
        \draw[red, very thick]  (c) to node[fill=white,inner sep = 2pt,rounded corners] {$8$} (p0);
        \draw[blue, very thick]  (a) to node[fill=white,inner sep = 2pt,rounded corners] {$9$} (p0);
        \foreach \node in {a,b,c,p0,x,y} \draw[fill] (\node) circle [radius=.09];
    \end{tikzpicture}
    \begin{tikzpicture}[scale=.65]
        \foreach \coordinate/\name in {(-2, 0)/a,(1, -1)/b,(1, 1)/c,(-2, 3)/p0,(-1, -3)/x,(2, -3)/y,(3, -3)/z} \node[shape=coordinate] at \coordinate (\name) {};
        \node[left] at (a) {$a$};
        \node[above right] at (c) {$c$};
        \node[below] at (b) {$b$};
        \node[above] at (p0) {$p_0$};
        \draw[red, very thick]  (a) to node[fill=white,inner sep = 2pt,rounded corners] {$0$} (x);
        \draw[blue, very thick]  (c) to node[fill=white,inner sep = 2pt,rounded corners] {$1$} (y);
        \draw[red, very thick]  (a) to node[fill=white,inner sep = 2pt,rounded corners] {$2$} (b);
        \draw[blue, very thick]  (x) to node[fill=white,inner sep = 2pt,rounded corners] {$3$} (b);
        \draw[red, very thick]  (a) to node[fill=white,inner sep = 2pt,rounded corners] {$4$} (c);
        \draw[blue, very thick]  (x) to node[fill=white,inner sep = 2pt,rounded corners,near end] {$5$} (c);
        \draw[red, very thick]  (c) to node[fill=white,inner sep = 2pt,rounded corners,near end] {$6$} (b);
        \draw[blue, very thick]  (z) to node[fill=white,inner sep = 0pt,rounded corners] {$7$} (y);
        \draw[red, very thick]  (c) to node[fill=white,inner sep = 2pt,rounded corners] {$8$} (p0);
        \draw[blue, very thick]  (a) to node[fill=white,inner sep = 2pt,rounded corners] {$9$} (p0);
        \foreach \node in {a,b,c,p0,x,y,z} \draw[fill] (\node) circle [radius=.09];
    \end{tikzpicture}
    \caption{The four cases in which \fp applies her strategy of \ref{item:stage_five}.}
    \label{fig:stage_5_special_cases}
\end{figure}
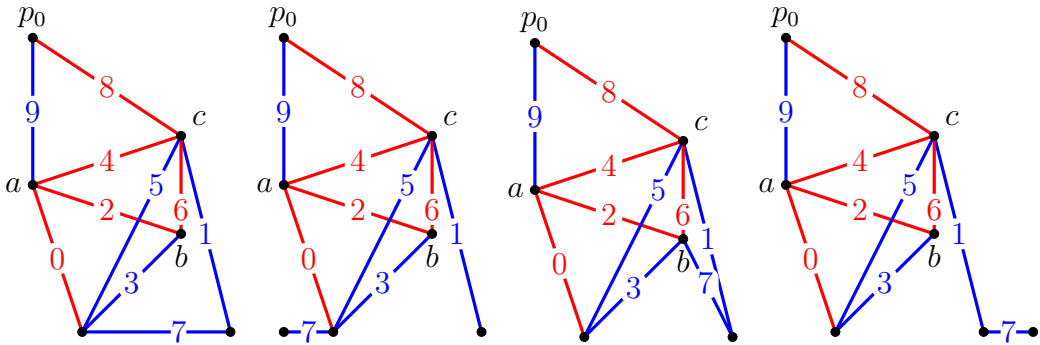

If after $5$ turns $\GF \cup \GS$ is isomorphic to any of the four graphs depicted in \cref{fig:stage_5_special_cases}, then \fp diverts from \labelcref{item:stage_two} of her regular strategy and instead plays as follows.
Her main goal is to either have a promising graph contained in the board as an induced subgraph or an independent set of five vertices in $\GF \cup \GS$ all with a common neighbour in $\GF$.
Let us call the vertices of the monochromatic red triangle in the board positions $a$, $b$ and $c$, where $c$ is the vertex of \fpdegree three and \secpdegree two, $a$ is the other vertex of \fpdegree 3 and $b$ is the last vertex of the triangle.
Furthermore let us give the name $p_0$ to the vertex of \fpdegree and \secpdegree one.
\begin{itemize}
    \item In her sixth turn \fp chooses a fresh vertex $p_1$ and claims $cp_1$.
    \item In her seventh turn, 
    \begin{itemize} \item if \secp did not colour $ap_1$ or $bp_1$ in his sixth turn, then \fp claims $ap_1$ (which is a threat to \secp in this case) and switches to \ref{item:stage_three} with the $\kvimi$ induced by $\braces{a,b,c,p_1}$ and the $p_i$ as assigned in this stage for $i \in \braces{0,1}$.
    \item Otherwise \secp claimed either $ap_1$ or $bp_1$ in his sixth turn.
        Then \fp chooses a fresh vertex $p_2$ and claims $cp_2$.
    \end{itemize}
    \item In \fps eighth turn,
    \begin{itemize}
        \item if \secp did not colour $ap_2$ or $bp_2$ or $p_1p_2$ in her seventh turn, then \fp claims $ap_2$ in her eighth turn and again switches to \ref{item:stage_three} with the $\kvimi$ induced by $\braces{a,b,c,p_2}$ and the $p_i$ as assigned in this stage for $i \in \braces{0,1,2}$.
        \item Otherwise \secp claimed either $ap_2$, $bp_2$ or $p_1p_2$ in his seventh turn.
            Then \fp chooses a fresh vertex $p_3$, claims $cp_3$ in her eighth turn and switches to \labelcref{item:stage_six}.
    \end{itemize}
\end{itemize}
\end{enumerate}
\begin{enumerate}[resume*]
\item\label{item:stage_six}
\begin{figure}[ht]
    \centering
    \begin{tikzpicture}[scale=.65]
        \foreach \coordinate/\name in {(-2, 0)/a,(1, -1)/b,(1, 1)/c,(-1, -3)/p0,(-2.5, -2)/x} \node[shape=coordinate] at \coordinate (\name) {};
        \node[below] at (p0) {$p_0$};
        \draw[red, very thick]  (a) to node[fill=white,inner sep = 2pt,rounded corners] {$0$} (p0);
        \draw[blue, very thick]  (a) to node[fill=white,inner sep = 2pt,rounded corners] {$1$} (x);
        \draw[red, very thick]  (a) to node[fill=white,inner sep = 2pt,rounded corners] {$2$} (b);
        \draw[blue, very thick]  (p0) to node[fill=white,inner sep = 2pt,rounded corners] {$3$} (b);
        \draw[red, very thick]  (a) to node[fill=white,inner sep = 2pt,rounded corners] {$4$} (c);
        \draw[blue, very thick]  (p0) to node[fill=white,inner sep = 2pt,rounded corners,near end] {$5$} (c);
        \draw[red, very thick]  (b) to node[fill=white,inner sep = 2pt,rounded corners] {$6$} (c);
        \draw[blue, very thick]  (p0) to node[fill=white,inner sep = 2pt,rounded corners] {$7$} (x);
        \foreach \node in {a,b,c,p0,x} \draw[fill] (\node) circle [radius=.09];
    \end{tikzpicture}
    \begin{tikzpicture}[scale=.65]
        \foreach \coordinate/\name in {(-2, 0)/a,(1, -1)/b,(1, 1)/c,(-1, -3)/p0,(-2.5, -2)/x,(1, -3)/y} \node[shape=coordinate] at \coordinate (\name) {};
        \node[below] at (p0) {$p_0$};
        \draw[red, very thick]  (a) to node[fill=white,inner sep = 2pt,rounded corners] {$0$} (p0);
        \draw[blue, very thick]  (a) to node[fill=white,inner sep = 2pt,rounded corners] {$1$} (x);
        \draw[red, very thick]  (a) to node[fill=white,inner sep = 2pt,rounded corners] {$2$} (b);
        \draw[blue, very thick]  (p0) to node[fill=white,inner sep = 2pt,rounded corners] {$3$} (b);
        \draw[red, very thick]  (a) to node[fill=white,inner sep = 2pt,rounded corners] {$4$} (c);
        \draw[blue, very thick]  (p0) to node[fill=white,inner sep = 2pt,rounded corners,near end] {$5$} (c);
        \draw[red, very thick]  (b) to node[fill=white,inner sep = 2pt,rounded corners] {$6$} (c);
        \draw[blue, very thick]  (p0) to node[fill=white,inner sep = 2pt,rounded corners] {$7$} (y);
        \foreach \node in {a,b,c,p0,x,y} \draw[fill] (\node) circle [radius=.09];
    \end{tikzpicture}
    \caption{The two cases in which \fp applies her strategy of \ref{item:stage_six}.}
    \label{fig:stage_6_special_cases}
\end{figure}
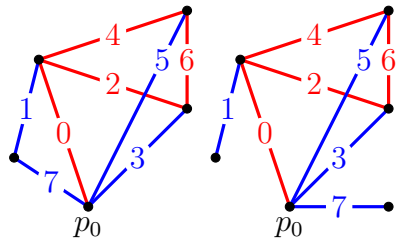
\FP may enter this stage in two different ways.
The first is that she enters this stage after exiting \labelcref{item:stage_five}.
The second is if after four turns $\GF \cup \GS$ is isomorphic to one of the two graphs depicted in \cref{fig:stage_6_special_cases}, then in \labelcref{item:stage_two} \fp diverts from her regular strategy and instead plays as follows.
Only in the latter case in the fifth turn she gives the name $p_0$ to the vertex of \fpdegree one and $a$ to the vertex of \fpdegree three.
In either case, from turn eight or turn six onwards respectively, \fp goes through the list of conditions below and executes the strategy of the first item from the list whose conditions are met.
\begin{enumerate}
    \item \FP checks whether she secured a way to win in a previous turn, defined from one of the options.
    \item \FP checks, whether there is a realised promising graph present as a subgraph of the board such that the corresponding promising graph fulfils the conditions of \cref{lem:promising_graphs}.
            If there is one of those present, then this lets her continue with a threat each turn until she finally wins as proved in \cref{lem:promising_graphs}.
            Thus she saves the information on how to continue for the next turns.
    \item \FP checks, whether there are five $p_i$ on the board and also checks whether there are no edges in between any two of them.
            Then she can build a triangle on the five vertices with the modified $\kthree$ building strategy.
            If this applies, she saves the information that she wants to continue to build the triangle on these five vertices.
    \item \FP checks, whether there are fewer than five vertices $p_i$, $i \in \braces{0,1,2,3,4}$.
            If this applies, she finds a fresh vertex, calls it $p_j$ where $j := \max{\braces{i + 1 \colon p_i \text{ is defined}}}$ and claims $cp_j$.
\end{enumerate}

\end{enumerate}
Note that one could get the impression that there can be courses of the game where \fp uses her strategy from \labelcref{item:stage_six} but the prerequisites are not met for any of the considered situations.
In practice this never occurs, as one is able to deduce from the proof.

\begin{theorem}\label{theo:k_four_game}
    The $K^4$ building strategy is a winning strategy for \fp in the $\kfour$ game.
    Moreover, \fp needs at most 21 moves to win the $\kfour$ game with the $\kfour$ building strategy.
\end{theorem}

\begin{proof}
    We prove this \namecref{theo:k_four_game} via the algorithm given in \texttt{K\_4\_game.py} and \texttt{functions.py}.
    There we implement the $\kfour$ building strategy in the function \texttt{FP\_edge()}, while \texttt{move()} recursively goes through every possible course of the game.
    For the exact implementation, we refer to the comments in the code.
    We do however make two modifications that are purely made for runtime considerations.\footnote{The runtime for the program with the adaptions but without multiprocessing on a MacBookPro with an Intel i5 1.4GHz Quad-Core processor is about 12h.}
    First, we make a list of possible non-isomorphic board states after \labelcref{item:stage_two}.
    \begin{claim}
        There are no additional courses of the game by considering other, isomorphic board states after \labelcref{item:stage_two}. \hfill $\rule{1.2ex}{1.2ex}$ \hspace{2pt}
    \end{claim}
    Second, in \labelcref{item:stage_three} we only consider that \secp claims the edge incident with \secpmain and $p_i$ for $i \in \braces{0,1,2}$  when \fp claimed the edge incident with \fpmain and $p_i$ in her turn and no other possible edge.
    Additionally, for $p_3$ we only allow that \secp claims an edge incident with $p_3$ and one of \secpmain, \secpleft, \secpright.
    \begin{claim}\label{claim:can_shorten_algorithm_ins_Stage_three}
        There are no additional courses of the game if \secp does not react to the $p_i$ in the implemented way in \ref{item:stage_three} up to isomorphism.
    \end{claim}
    \begin{claimproof}
        Suppose that it is \secps turn and \fp chose a fresh vertex $p_i$ and claimed $p_i \fpmainInMath $ in her previous turn.
        In each of the mentioned turns, if \secp can not create a threat somewhere on the board himself, he must claim an edge incident with $p_i$ and one of \secpmain, \secpleft, \secpright because otherwise \fp can win in two turns by claiming the edge $p_i \secpmainInMath $.
        If he can create a threat somewhere else on the board, he can also do so at the beginning of \labelcref{item:stage_four} and since \fp claims edges incident with vertices that are fresh at some point during \labelcref{item:stage_three}, edges claimed by \fp during \labelcref{item:stage_three} will not mitigate that threat.
        Thus, if anything \secp has more opportunities to create threats at the beginning of \labelcref{item:stage_four}.

        Now suppose that at some turn during \labelcref{item:stage_three} \secp claims $p_i\secpleftInMath$.
        Then \fp claims $p_i\secprightInMath$ in her following turn.
        This creates a threat to which \secp must react by claiming $p_i\secpmainInMath$.
        For a $j \neq i$ \secp cannot claim $p_j \secpleftInMath$ any more because by then claiming $p_j \secprightInMath$, \fp creates two threats and thus wins the game.

        For a similar reason \secp can claim an edge incident with a $p_i$ and \secpright at most once.

        Thus \secp can claim an edge incident with a $p_i$ and \secpleft at most once and only an edge incident with a different $p_j$ and \secpright at most once.
        After that he must always claim the edge $p_k \secpmainInMath $ right after \fp claimed $p_k \fpmainInMath$.
        Thus any possible board state after \labelcref{item:stage_three} will be isomorphic to one we implemented in the algorithm, apart from threats by \secp.
        This, together with the fact that this does not infringe on the possibilities for \secp to create threats proves the \namecref{claim:can_shorten_algorithm_ins_Stage_three}.
    \end{claimproof}
    This finishes the proof.
\end{proof}

\end{document}